\documentclass[reqno]{amsart}
\usepackage{amssymb}
\usepackage{ifpdf}
\ifpdf
 \usepackage[hyperindex,pagebackref]{hyperref}
\else
 \expandafter\ifx\csname dvipdfm\endcsname\relax
 \usepackage[hypertex,hyperindex,pagebackref]{hyperref}
 \else
 \usepackage[dvipdfm,hyperindex,pagebackref]{hyperref}
 \fi
\fi
\allowdisplaybreaks[4]
\numberwithin{equation}{section}
\theoremstyle{plain}
\newtheorem{theorem}{Theorem}[section]
\newtheorem{open}{Open Problem}[section]
\newtheorem{lemma}{Lemma}[section]
\theoremstyle{remark}
\newtheorem{remark}{Remark}[section]
\DeclareMathOperator{\td}{d\mspace{-1mu}}
\DeclareMathOperator{\tdq}{d_{\text{$q$}}\mspace{-1mu}}

\begin{document}

\title{Some integral inequalities on time scales}

\author[L. Yin]{Li Yin}
\address[L. Yin]{Department of Mathematics, Binzhou University, Binzhou City, Shandong Province, 256603, China}
\email{\href{mailto: L. Yin<yinli_79@163.com>}{yinli\_79@163.com}}

\author[F. Qi]{Feng Qi}
\address[F. Qi]{Department of Mathematics, College of Science, Tianjin Polytechnic University, Tianjin City, 300160, China; School of Mathematics and Informatics, Henan Polytechnic University, Jiaozuo City, Henan Province, 454010, China}
\email{\href{mailto: F. Qi<qifeng618@gmail.com>}{qifeng618@gmail.com}, \href{mailto: F. Qi<qifeng618@hotmail.com>}{qifeng618@hotmail.com}, \href{mailto: F. Qi<qifeng618@qq.com>}{qifeng618@qq.com}}
\urladdr{\url{http://qifeng618.wordpress.com}}

\subjclass[2010]{Primary 26D15, 26E70; Secondary 93C70}

\keywords{integral inequality; calculus of time scales; delta differentiability; H\"older inequality; unified form}

\begin{abstract}
In this paper, some new integral inequalities on time scales are presented by using elementarily analytic methods in calculus of time scales.
\end{abstract}

\thanks{This paper was typeset using\AmS-\LaTeX}

\maketitle

\section{ Introduction}

The following problem was posed by the second author in the preprint~\cite{Qi-several-RGMIA-99-1042} and its formally published version~\cite{Qi-several}.
\begin{open}\label{Qi-int-ineq-open}
Under what conditions does the inequality
\begin{equation}\label{Yin-Qi-ineq-open}
\int_a^b{[f(x)]^p\td x\ge}\biggl[\int_a^b{f(x)\td x}\biggr]^{p-1}
\end{equation}
hold for $p>1$?
\end{open}

Since then, this problem has been stimulating much interest of many mathematicians. In recent years, we have collected over forty articles devoted to answering and generalizing this open problem and to applying inequalities of this type. For potential availability to interested readers, we list the collection as references of this paper.
\par
In~\cite[p.~124, Theorem~C]{akk-several}, M. Akkouchi proved the following result.

\begin{theorem}\label{Akkouchi-thm}
Let $[a,b]$ be a closed interval of $\mathbb{R}$ and $p>1$. If $f(x)$ is a continuous function on $[a,b]$ such that $f(a)\ge0$ and $f'(x)\ge{p}$ on $(a,b)$, then
\begin{equation}\label{Yin-Qi-ineq-akk}
\int_a^b{[f(x)]^{p+2}\td x\ge}\frac{1}{{(b-a)^{p-1}}}\biggl[\int_a^b{f(x)\td x}\biggr]^{p+1}.
\end{equation}
\end{theorem}

In~\cite[Proposition~3.5]{Brahim-Bett-Sell-08}, the following $q$-analogue of the above Theorem~\ref{Akkouchi-thm} was obtained.

\begin{theorem}\label{Brahim-Bett-Sell-08-thm}
Let $p>1$ be a real number and $f(x)$ be a function defined on $[a,b]_q$ such that $f(a)\ge0$ and $\tdq{f(x)}\ge{p}$ for all $x\in(a,b]_q$. Then
\begin{equation}\label{Yin-Qi-ineq-bra}
\int_a^b{[f(x)]^{p+2}\tdq x\ge}\frac{1}{{(b-a)^{p-1}}}\biggl[\int_a^b{f(qx)\tdq x}\biggr]^{p+1}.
\end{equation}
\end{theorem}

The main aim of this paper is to generalize the above results on time scales. As by-product, an unified form of Theorems~\ref{Akkouchi-thm} and~\ref{Brahim-Bett-Sell-08-thm} is demonstrated when $p\ge2$.

\section{Notations and lemmas}

\subsection{Notations}

Throughout this paper, we adopt notations in the monograph~\cite{Bohner-Peterson-01-B}.
\par
A time scale $\mathbb{T}$ is a non-empty closed subset of the real numbers $\mathbb{R}$. The forward and backward jump operators $\sigma, \rho :\mathbb{T}\to\mathbb{T}$ are respectively defined by
\begin{equation}
\sigma(t)=\inf\{{s\in\mathbb{T}:s>t}\}
\end{equation}
and
\begin{equation}
\rho(t)=\sup\{{s\in\mathbb{T}:s<t}\},
\end{equation}
where the supremum of the empty set is defined to be the infimum of $\mathbb{T}$.
\par
A point $t\in\mathbb{T}$ is said to be right\nobreakdash-scattered if $\sigma(t)>t$ and to be right\nobreakdash-dense if $\sigma(t)= t$; on the other hand, a point $t\in\mathbb{T}$ with $ t>\inf\mathbb{T}$ is said to be left\nobreakdash-scattered if $\rho(t)<t$ and to be left\nobreakdash-dense if $\rho(t)=t$.
A function $g:\mathbb{T}\to\mathbb{R}$ is said to be rd\nobreakdash-continuous provided that $g$ is continuous at right\nobreakdash-dense points and has finite left\nobreakdash-sided limits at left\nobreakdash-dense points in $\mathbb{T}$. In what follows, the set of all rd\nobreakdash-continuous functions from $\mathbb{T}$ to $\mathbb{R}$ is denoted by $C_{rd}(\mathbb{T}, \mathbb{R})$. The graininess function $\mu$ for a time scales $\mathbb{T}$ is defined by $\mu(t)=\sigma(t)-t$. For $f:\mathbb{T}\to\mathbb{R}$, the notation $f^\sigma$ means the composition $f\circ\sigma$.
\par
For a function $f:\mathbb{T}\to\mathbb{R}$, the (delta) derivative $f^\Delta(t)$ at $t\in\mathbb{T}$ is defined to be the number, if it exists, such that for all $\varepsilon >0$, there is a neighborhood $U$ of $t$ with
\begin{equation}
\bigl|f(\sigma(t))-f(s)-f^\Delta(t)(\sigma(t)-s)\bigr|<\varepsilon |\sigma(t)-s|
\end{equation}
for all $s\in U$. If the (delta) derivative $ f^\Delta(t)$ exits for all $t\in\mathbb{T}$, then we say that $f$ is (delta) differentiable on $\mathbb{T}$.
\par
The (delta) derivatives of the product $fg$ and the quotient $\frac{f}{g}$ of two (delta) differentiable functions $f$ and $g$ may be formulated respectively as
\begin{equation}
(fg)^\Delta=f^\Delta  g+f^\sigma  g^\Delta=fg^\Delta+f^\Delta  g^\sigma
\end{equation}
and
\begin{equation}
\biggl({\frac{f}{g}}\biggr)^\Delta=\frac{{f^\Delta  g-fg^\Delta }}{{gg^\sigma }},
\end{equation}
where $gg^\sigma\ne0$.
\par
For $ b, c\in\mathbb{T}$ and a (delta) differentiable function $f$, Cauchy integral of $f^\Delta$ is defined by
\begin{equation}
\int_b^c{f^\Delta(t)\Delta t=f(c)-f(b)}
\end{equation}
and infinite integrals are defined as
\begin{equation}
\int_b^\infty {f(t)\Delta t=\mathop{\lim}_{c\to\infty}}\int_b^c{f(t)}\Delta{t}.
\end{equation}
An integration by parts formula reads that
\begin{equation}
\int_b^c{f(t)g^\Delta(t)\Delta t=f(t)g(t)\big|_b^c-\int_b^c{f^\Delta(t)g(\sigma(t))\Delta t}}.
\end{equation}
Note that in the case $\mathbb{T}=\mathbb{R}$, we have $\sigma(t)=\rho(t)=t$, $\mu(t)=0$, $f^\Delta(t)= f'(t)$, and
\begin{equation}
\int_b^c{f^\Delta(t)\Delta t=\int_b^c{f'(t)\td t}},
\end{equation}
and that in the case $\mathbb{T}=q\mathbb{Z}$, we have $\sigma(t)=t+q$, $\rho(t)=t-q$, $\mu(t)\equiv q$, and
\begin{equation}
f^\Delta(t)=\frac{ f(t+q)-f(t)}{q}.
\end{equation}

\subsection{Lemmas}

The following lemmas are useful and some of them can be found in the book~\cite{Bohner-Peterson-01-B}.

\begin{lemma}[{\cite[p.~28, Theorem~1.76]{Bohner-Peterson-01-B}}]\label{decrease-lem}
If $f^{\Delta}(x)\ge0$, then $f(x)$ is non-decreasing.
\end{lemma}

\begin{lemma}[{\cite[p.~32, Theorem~1.90]{Bohner-Peterson-01-B}, Chain Rule}]\label{Chain-Rule-lem}
Let $f:\mathbb{R}\to\mathbb{R}$ be continuously differentiable and let $g:\mathbb{T}\to\mathbb{R}$ be delta differentiable. Then $f\circ{g}:\mathbb{T}\to\mathbb{R}$ is delta differentiable and
\begin{equation}\label{Yin-Qi-ineq-2.11}
(f\circ g)^\Delta(t)=\biggl\{\int_0^1{\bigl[f'(g(t)+h\mu
(t)g^\Delta }(t)\bigr]\td h\biggr\}g^\Delta(t).
\end{equation}
\end{lemma}

\begin{lemma}[{\cite[p.~34, Theorem~1.93]{Bohner-Peterson-01-B}}]\label{2.12-lem}
Assume that $v:\mathbb{T}\to{\mathbb{R}}$ is strictly increasing, that $\mathbb{\tilde{T}}=v(\mathbb{T})$ is a time scale, and that $\omega:\mathbb{\tilde{T}}\to{\mathbb{R}}$. If $v^\Delta(t)$ and $\omega^{\tilde{\Delta}}(v(t))$ exist for $t\in\mathbb{T}$, then
\begin{equation}\label{Yin-Qi-ineq-2.12}
(\omega \circ v)^\Delta =\Bigl(\omega^{\tilde{\Delta}}\circ v\Bigr)v^{\Delta}.
\end{equation}
\end{lemma}

\begin{lemma}[{\cite[p.~259, Theorem~6.13]{Bohner-Peterson-01-B}, H\"older inequality}]\label{holder-ineq-lem}
Let $a,b\in\mathbb{T}$. If $f,g\in C_{rd}(\mathbb{T}, \mathbb{R})$, then
\begin{equation}\label{Yin-Qi-ineq-2.13}
\int_a^b{|f(x)g(x)|\Delta x\le}\biggl[\int_a^b |f(x)|^p\Delta x\biggr]^{1/p}
\biggl[\int_a^b |g(x)|^q\Delta x\biggr]^{1/q},
\end{equation}
where $p>1$ and $\frac{1}{p}+\frac{1}{q}=1$.
\end{lemma}

\begin{lemma}\label{holder-ineq-deduction-lem}
 Let $a, b\in\mathbb{T}$. If $f,g\in C_{rd}(\mathbb{T},
\mathbb{R})$ are positive, then
\begin{equation}\label{Yin-Qi-ineq-2.14}
\int_a^b{\frac{{[f(x)]^p}}
{{[g(x)]^{p/q}}}\Delta x\ge}\frac{{\bigl[\int_a^b{f(x)\Delta x}\bigr]^p}}{{\bigl[\int_a^b{g(x)\Delta x}\bigr]^{p/q}}},
\end{equation}
where $p>1$ or $p<0$ while $\frac{1}{p}+\frac{1}{q}=1$.
\par
The equality in~\eqref{Yin-Qi-ineq-2.14} holds
if and only if $f(x)=ag(x)$ for $a>0$.
\end{lemma}

\begin{proof}
For $p>1$, using the inequality~\eqref{Yin-Qi-ineq-2.13} in Lemma~\ref{holder-ineq-lem} leads to
\begin{equation*}
\int_a^b f(x)\Delta x=\int_a^b\frac{{f(x)}}
{{\sqrt[q]{{g(x)}}\,}}\sqrt[q]{{g(x)}}\, \Delta x
\le \biggl(\int_a^b{\frac{{[f(x)]^p}}{{[g(x)]^{p/q}}}\Delta x}\biggr)^{\frac{1}
{p}}\biggl(\int_a^b{g(x)\Delta x}\biggr)^{1/q}.
\end{equation*}
Further taking the $p$-th power on both sides of the above inequality yields~\eqref{Yin-Qi-ineq-2.14}.
\par
For $p<0$, utilizing the inverse of H\"older's inequality and similar argument as a little time ago result in the required inequality.
\end{proof}

\begin{lemma}\label{holder-ineq-ded-lem}
Let $a, b\in\mathbb{T}$. If $f,g\in C_{rd}(\mathbb{T}, \mathbb{R})$ and
\begin{equation}
0<m\le{\frac{f(x)}{g(x)}}\le{M}<\infty,
\end{equation}
then we have
\begin{equation}\label{Yin-Qi-ineq-2.15}
\biggl[\int_a^b{f(x)\Delta x}\biggr]^{1/p}\biggl[\int_a^b{g(x)\Delta x}
\biggr]^{1/q}\le\biggl(\frac{M}{m}\biggr)^{1/pq}
\int_a^b{f^{1/p}(x)g^{1/q}(x)\Delta x},
\end{equation}
where $p>1$ and $\frac{1}{p}+\frac{1}{q}=1$.
\end{lemma}

\begin{proof}
From $\frac{f(x)}{g(x)}\le{M}$, it follows that
$$
g^{1/q}(x)\ge M^{-1/q}f^{1/q}(x)
$$
and
\begin{equation*}
f^{1/p}(x)g^{1/q}(x)\ge M^{-1/q}f^{1/q}(x)f^{1/p}(x)=M^{-1/q}f(x).
\end{equation*}
Accordingly,
\begin{equation}\label{ineq-Yin-3.5}
\biggl[\int_a^b
{f^{1/p}(x)g^{1/q}(x)\Delta x}\biggr]^{1/p}
\ge M^{-1/pq}\biggl[\int_a^b{f(x)\Delta x}\biggr]^{1/p}.
\end{equation}
On the other hand, from $\frac{f(x)}{g(x)}\ge{m}$, it follows easily that
\begin{equation}\label{ineq-Yin-3.6}
\biggl[\int_a^b{f^{1/p}(x)g^{1/q}
(x)\Delta x}\biggr]^{1/q}\ge m^{1/pq}
\biggl[\int_a^b{g(x)\Delta x}\biggr]^{1/q}.
\end{equation}
Multiplying~\eqref{ineq-Yin-3.5} and~\eqref{ineq-Yin-3.6} leads to~\eqref{Yin-Qi-ineq-2.15}.
\end{proof}

Replacing $f(x)$ and $g(x)$ respectively by $f^{p}(x)$ and $g^{q}(x)$ in Lemma~\ref{holder-ineq-ded-lem}, it is immediate to obtain the following conclusion.

\begin{lemma}\label{holder-ineq-final-lem}
Let $a, b\in\mathbb{T}$. If $f,g\in C_{rd}(\mathbb{T}, \mathbb{R})$ and
$$
0<m\le{\frac{f^{p}(x)}{g^{q}(x)}}\le{M}<\infty,
$$
then
\begin{equation}\label{Yin-Qi-ineq-2.16}
\biggl[\int_a^b{f^{p}(x)\Delta x}\biggr]^{1/p}\biggl[\int_a^b{g^{q}(x)\Delta x}\biggr]^{1/q} \le\biggl(\frac{M}{m}\biggr)^{1/pq}\int_a^b{f(x)g(x)\Delta x},
\end{equation}
where $p>1$ and $\frac{1}{p}+\frac{1}{q}=1$.
\end{lemma}

\begin{remark}
When $\mathbb{T}=\mathbb{R}$, Lemma~\ref{holder-ineq-final-lem} becomes~\cite[Theorem~2.1]{Saitoh-Tuan-Yamamoto-jipam-02-art80}.
\end{remark}

\section{ Main Results}

Now we are in a position to state and prove our main results.

\begin{theorem}\label{Yin-Qi-thm-1}
Let $a, b\in\mathbb{T}$. If $f\in C_{rd}(\mathbb{T}, \mathbb{R})$ is positive and
\begin{equation}\label{Yin-Qi-ineq-3.1}
\int_a^b{f(x)\Delta x\ge(b-a)^{p-1},}
\end{equation}
then
\begin{equation}\label{Yin-Qi-ineq-3.2}
\int_a^b{[f(x)]^p\Delta{x}\ge}\biggl[\int_a^b{f(x)\Delta{x}}\biggr]^{p-1},
\end{equation}
where $p>1$ or $p<0$.
\end{theorem}

\begin{proof}
Using Lemma~\ref{holder-ineq-deduction-lem} and the condition~\eqref{Yin-Qi-ineq-3.1}, we obtain
\begin{equation*}
\int_a^b{[f(x)]^p\Delta x=\int_a^b{\frac{{[f(x)]^p}}{{1^{p-
1}}}\Delta x}\ge}\frac{{\bigl[\int_a^b{f(x)\Delta x}\bigr]^p}}
{{\bigl[\int_a^b{1\Delta x}\bigr]^{p-1}}}
\ge\biggl[\int_a^b{f(x)\Delta x}\biggr]^{p-1}.
\end{equation*}
The proof of Theorem~\ref{Yin-Qi-thm-1} is complete.
\end{proof}

\begin{remark}
Letting $\mathbb{T}=\mathbb{R}$ and $p>1$ in Theorem~\ref{Yin-Qi-thm-1}, we deduce~\cite[Theorem~A]{Liu-Ngo-Huy-JMI-09-212}.
\end{remark}

\begin{theorem}\label{Yin-Qi-thm-2}
Let $a, b\in\mathbb{T}$. If $f(x)$ and $\sigma(x)$ are rd\nobreakdash-continuous and positive on $[a,b]$, $\Delta$-differentiable on $(a,b)$, $f(a)\ge{\mu(a)}\ge0$, and $ f^\Delta(x)\ge 1+
\sigma ^\Delta(x)$ for all $x\in(a,b)$, then the inequality
\begin{equation}\label{Yin-Qi-ineq-3.3}
\int_a^b{[f(x)]^{p+2}\Delta x\ge}\frac{1}{{(b-a)^{p-1}}}\biggl[\int_a^b{f(x)\Delta x}\biggr]^{p+1}
\end{equation}
is valid for $p\ge1$.
\end{theorem}

\begin{proof}
By Lemma~\ref{holder-ineq-deduction-lem}, we obtain
\begin{align*}
\int_a^b [f(x)]^{p+2}\Delta x &=\int_a^b{\frac{{[f(x)]^{p+2}}}{{1^{p+1}}}\Delta x}\\*
&\ge \frac{{\bigl[\int_a^b{f(x)\Delta x}\bigr]^{p+2}}}{{\bigl(\int_a^b{1\Delta x}\bigr)^{p+1}}}\\
&=\frac{{\bigl[\int_a^b{f(x)\Delta x}\bigr]^{p+1}}}{{(b-a)^{p-1}}} \cdot\frac{{\int_a^b{f(x)\Delta x}}}{{(b-a)^2}}.
\end{align*}
So it is enough to show $\int_a^b{f(x)\Delta x\ge(b-a)^2}$. For this, let
$$
F(x)=\int_a^x{f(t)\Delta t-(x-a)^2}.
$$
A simple computation yields
\begin{equation*}
F^\Delta(x)=f(x)-[x-a+\sigma(x)-a]
\end{equation*}
and
\begin{equation*}
F^{\Delta\Delta}(x)=f^\Delta(x)-1-\sigma ^\Delta(x).
\end{equation*}
Since $F^{\Delta\Delta}(x)\ge 0$ and
$$
F^\Delta(a)=f(a)-[\sigma(a)-a]=f(a)-\mu(a)\ge 0,
$$
it follows that $ F^\Delta(x)
\ge F^\Delta(a)\ge 0$. Therefore, by Lemma~\ref{decrease-lem}, the function $F(x)$ is
increasing. From $F(a)=0$, it is easy to see that $F(x)\ge{F(a)}=0$. The proof is complete.
\end{proof}

\begin{remark}
If $\mathbb{T}=\mathbb{R}$ and $p\ge2$, then $\sigma(x)=x$, $\sigma^{\Delta}(x)=1$, $f(a)\ge0$, and $f^{\Delta}(x)\ge{p}\ge2$. Then Theorem~\ref{Akkouchi-thm} may be concluded from Theorem~\ref{Yin-Qi-thm-2}.
\end{remark}

\begin{remark}
If $\mathbb{T}=q\mathbb{Z}$ and $p\ge2$, then $\sigma(x)=x+q$, $\sigma^{\Delta}(x)=1$, $f(a)\ge{q}>0$, and $f^{\Delta}(x)\ge{p}\ge2$. Consequently, Theorem~\ref{Yin-Qi-thm-2} becomes Theorem~\ref{Brahim-Bett-Sell-08-thm}.
\end{remark}

\begin{remark}
When $p\ge2$, Theorem~\ref{Yin-Qi-thm-2} unifies both of Theorems~\ref{Akkouchi-thm} and~\ref{Brahim-Bett-Sell-08-thm}.
\end{remark}

\begin{theorem}\label{Yin-Qi-thm-3}
Let $a, b\in\mathbb{T}$ and $p>1$ with $\frac{1}{p}+\frac{1}{q}=1$. If $0<m\le{f^{p}(x)}\le{M}<\infty$ on $[a,b]$, then
\begin{equation}\label{Yin-Qi-ineq-3.4}
\biggl[\int_a^b{f^p(x)\Delta x}\biggr]^{1/p}\le(b-a)^{-(p+ 1)/q}\biggl(\frac{M}{m}\biggr)^{2/pq}\biggl[\int_a^b{f^{1/p}(x)\Delta x}\biggr]^p.
\end{equation}
\end{theorem}

\begin{proof}
Putting $g(x)=1$ into Lemma~\ref{holder-ineq-final-lem} yields
\begin{gather}
\biggl[\int_a^b{f^p(x)\Delta x}\biggr]^{1/p}\biggl(\int_a^b{1\Delta x}
\biggr)^{1/q}\le\biggl(\frac{M}{m}\biggr)^{1/pq}
\int_a^b{f(x)\Delta x}, \notag\\
\biggl[\int_a^b{f^p(x)\Delta x}\biggr]^{1/p}\le\biggl(\frac{M}{m}\biggr)^{1/pq}(b- a)^{-1/q}\int_a^b{f(x)\Delta x}. \label{YIn-3.12}
\end{gather}
Substituting $g(x)=1$ in Lemma~\ref{holder-ineq-ded-lem} leads to
\begin{gather}
\biggl[\int_a^b{f(x)\Delta x}\biggr]^{1/p}\le(b-a)^{-1/q}
\biggl(\frac{M}{m}\biggr)^{1/p^2 q}\int_a^b{f^{1/p}
(x)\Delta x}, \notag\\
\int_a^b{f(x)\Delta x}\le(b-a)^{-p/q}
\biggl(\frac{M}{m}\biggr)^{1/pq}\biggl[\int_a^b{f^{1/p}
(x)\Delta x}\biggr]^p. \label{YIn-3.15}
\end{gather}
Combining \eqref{YIn-3.12} with \eqref{YIn-3.15}, the inequality~\eqref{Yin-Qi-ineq-3.4} follows.
\end{proof}

\begin{theorem}\label{Yin-Qi-thm-4}
If $f(x)$ is $\Delta$-differentiable on the closed interval $[a,b]$, $f(a)=0$, and $0<f^{\Delta}(x)<1$ for all $x\in(a,b)$, then the inequality
\begin{equation}\label{Yin-Qi-ineq-3.5}
2^{1-p}p\int_a^b f^{2p-1}(x)\Delta{x}<
\biggl[\int_a^b f(x)\Delta{x}\biggr]^p,
\end{equation}
holds for $p>1$.
\end{theorem}

\begin{proof}
Let
\begin{equation*}
 F(x)=\biggl[\int_a^x{f(t)\Delta{t}}\biggr]^p-2^{1-p}p\int_a^x{f^{2p-1}
(t)\Delta{t}}, \quad x\in(a,b).
\end{equation*}
Applying Lemma~\ref{Chain-Rule-lem} and differentiating $F(x)$ with respect to $x$ give
\begin{align*}
F^\Delta(x) &=pg^\Delta(x)\int_0^1{\bigl[g(x)+\mu(x)hg^\Delta(x)\bigr]^{p-1}}\td h-p2^{1-p}f^{2p-1}(x) \\
 &\ge pg^\Delta(x)\int_0^1{[g(x)]^{p-1}}\td h-p2^{1-p}f^{2p-1}(x)\\
 &=pf(x)[g(x)]^{p-1}-p2^{1-p}f^{2p-1}(x),
\end{align*}
where $g(x)=\int_a^x{f(t)\Delta t}$ and $g^{\Delta}(x)=f(x)$.
\par
Since $f(a)=0$ and $f^{\Delta}(x)>0$, it is easy to see that $f(x)>f(a)=0$.
\par
It is clear that the inequality
$$
\biggl[\int_a^x{f(t)\Delta{t}}\biggr]^{p-1}-2^{1-p}f^{2p-2}(x)>0
$$
is equivalent to
\begin{equation*}
\biggl[\int_a^x{f(t)\Delta{t}}\biggr]^{p-1}>\biggl[\frac{f^2(x)}{2}\biggr]^{p-1}, \quad p>1.
\end{equation*}
Let
$$
G(x)=\int_a^x{f(t)\Delta{t}}-\frac{f^2(x)}{2}.
$$
Then, by Lemma~\ref{2.12-lem}, we have
\begin{align*}
G^{\Delta}(x)&=f(x)-\frac{1}{2}[f(x)+f^{\sigma}(x)]f^{\Delta}(x)\\*
&\ge{f(x)-f(x)f^{\Delta}(x)}\\
&=f(x)\bigl[1-f^{\Delta}(x)\bigr]\\
&>0.
\end{align*}
Taking into account $G(0)=0$, the inequality $G(x)>0$ follows. Thus, the proof of Theorem~\ref{Yin-Qi-thm-4} is complete.
\end{proof}

\begin{remark}
If $\mathbb{T}=\mathbb{R}$, Theorem~\ref{Yin-Qi-thm-4} reduces to~\cite[Lemma~3.5]{L-Yin-Creative-11-95}. If $\mathbb{T}=\mathbb{R}$ and $p=2$, Theorem~\ref{Yin-Qi-thm-4} becomes~\cite[Proposition 1.1]{Qi-several}.
\end{remark}

\end{document}